\documentclass[10pt, article]{amsart}

\usepackage{tikz}
\usetikzlibrary{calc}
\usepackage{ae} 
\usepackage[T1]{fontenc}
\usepackage[cp1250]{inputenc}
\usepackage{amsmath}
\usepackage{amssymb, amsfonts,amscd,verbatim}

\usepackage[normalem]{ulem}
\usepackage{hyperref}
\usepackage{indentfirst}
\usepackage{latexsym}
\input xy
\xyoption{all}

\usepackage{amsmath}    

\theoremstyle{plain}
\newtheorem{Pocz}{Poczatek}[section]
\newtheorem{Proposition}[Pocz]{Proposition}
\newtheorem{Theorem}[Pocz]{Theorem}
\newtheorem{Corollary}[Pocz]{Corollary}
\newtheorem{Axiom}[Pocz]{Axiom}
\newtheorem{Axioms}[Pocz]{Axioms}
\newtheorem{Law}[Pocz]{Law}

\newtheorem{Notation}[Pocz]{Notation}
\newtheorem{Question}[Pocz]{Question}
\newtheorem{Problem}[Pocz]{Problem}
\newtheorem{Example}[Pocz]{Example}

\theoremstyle{definition}
\newtheorem{Definition}[Pocz]{Definition}

\theoremstyle{remark}
\newtheorem{Remark}[Pocz]{Remark}
\newtheorem{Exercise}[Pocz]{Exercise}

\numberwithin{equation}{section}
\title[
An approach to basic set theory and logic
]
   {An approach to basic set theory and logic}

\author{Jerzy Dydak}

\address{University of Tennessee, Knoxville, TN 37996}
\email{jdydak\@@utk.edu}

\date{ \today
}


\begin{document}
\maketitle
\begin{center}
\today
\end{center}

\begin{abstract}
The purpose of this paper is to outline a simple set of axioms for basic set theory from which most fundamental facts can be derived. The key to the whole project is a new axiom of set theory which I dubbed \lq The Law of Extremes\rq . It allows for quick proofs of basic set-theoretic identities and logical tautologies, so it is also a good tool to aid one's memory.

I do not assume any exposure to euclidean geometry via axioms. Only an experience with transforming algebraic identities is required.

The idea is to get students to do proofs right from the get-go. In particular, I avoid entangling students in nuances of logic early on. Basic facts of logic are derived from set theory, not the other way around.
\end{abstract}

\section{Introduction}

There is a widespread perception that students entering universities these days lack mathematical reasoning. As a result they are rushed through Introduction to Abstract Mathematics courses in the hope they can catch up and take serious math courses as juniors and seniors. My personal experience is that getting to that level requires longer time. I grew up in a system where rigorous Euclidean planar geometry was taught in the ninth grade, continued with stereometry in the tenth grade, and culminating with trigonometry in the eleventh grade. This was in addition to having algebra every year, with scarcity of proofs in that subject.
Logic began at the university level with a course in Foundations of Mathematics. Even with a few years of doing a significant amount of proofs, many students struggled with that level of abstractness. I vividly remember students who liked Euclidean geometry but were put off by incantations of injection, surjection, bijection.

My point is that learning abstract logic and set theory is very hard and cramming a lot of material in a one semester course on introduction to abstract mathematics is not going to do much good. I want students to enjoy the axiomatic method the same way I did: I was given axioms and lots of problems to do on my own, not to regurgitate their solutions rapidly.

The purpose of this paper is to outline a way to do it. It is not as good as planar Euclidean geometry but it offers a few very basic axioms and invites students to do problems immediately and on their own.

Here is the \textbf{big picture} of what I do:\\
1. First, basic set identities are codified for the simplest interesting case (subsets of a one-point set). This leads to the standard model for minimal universe $SMU=\{\emptyset\}$ (or $SMU=\{0\}$ if one uses algebraic notation $0$ for the empty set). One should think of $SMU$ as the Bing Bang in physics: all basic laws of set theory and logic are present in $SMU$ and do not change with the expansion of the universe.\\
2. For any universe $U$, a path from $U$ to $SMU$ is provided in order to generate identities in $U$ (Identity Generating Machine \ref{LawGenerator}). The path $U\to SMU$ is a form of \textbf{universe contraction} as opposed to \textbf{universe expansion} from $U$ to $2^U$.\\
3. Logic (calculus of statements) is built based on set-theoretic concepts.\\
4. A path from logic to set-theory of $SMU$ is provided. A path that reduces tautologies to set-theoretic identities.\\
5. A path from set theory to logic is provided. A path that reduces set-theoretic identities to logical tautologies.

The most interesting axiom in 1) (The Law of Extremes) allows for proving of basic set-theoretic identities rather quickly and helps in recovering those identities if they are needed in a more complicated proof. The idea is not to get students bogged down in essentially circular reasoning but to direct them to the most interesting proofs in the beginning of anyone's mathematical education:\\
i. $\sqrt{2}$ being irrational;\\
ii. Existence of infinitely many primes;\\
iii. Uncountability of irrational numbers;\\
iv. Set-theoretic paradoxes.

Here is a specific example of what I am trying to accomplish: instead of training students how to prove basis identities of set-theory (example: $A\cap B=B\cap A$, $A\cup B=B\cup A$, $A\cap (B\cup C)=(A\cap B)\cup (A\cap C)$, etc)
using logic, I establish those identities quickly from basic axioms (The Law of Extremes being the most important one). Only then logic is introduced and students are directed to ponder really interesting questions (example: what functions preserve which set-theoretic identities). 

I am grateful to Chuck Collins for conversations leading to improvements in the paper.

\section{Prerequisites}

There are three skills needed to embark on learning basic set theory and logic from these notes. One is understanding of basic algebraic definitions (first four problems below).

\begin{Problem}
Assume only the knowledge of addition and multiplication of
real numbers.
 What do we mean by $u^{-1}=v$? Can this equality be verified?
\end{Problem}

\begin{Problem}
Assume only the knowledge of addition and multiplication of
real numbers.
 What do we mean by $\frac{u}{v}=w$? Can this equality be verified?
\end{Problem}

\begin{Problem}
Suppose $f(x)$ is a given real-valued function.
 What do we mean by the statement: 
$$ \textrm{Solve }f(x)=0?$$
\end{Problem}

\begin{Problem}
Assume only the knowledge of addition and multiplication of
real numbers.
 What do we mean by $\sqrt{u}=v$? Can this equality be verified?
\end{Problem}

The next five problems are designed to test implimentation of basic definitions (students who simply follow algorithms usually have wrong solutions to those problems).

\begin{Problem}
Solve $(3\cdot x-9)^{-1}=(4\cdot x-12)^{-1}$. Show your work.
\end{Problem}

\begin{Problem}
Solve $(2\cdot x-4)^{-1}=(3\cdot x-6)^{-1}$. Show your work.
\end{Problem}

\begin{Problem}
Solve $x^2=8^{2}-2\cdot 8\cdot 3 +3^{2}$ without the square root function. Show your work.
\end{Problem}

\begin{Problem}
Solve $(2\cdot x-8)^{1/2}=(x-5)^{1/2}$. Show your work.
\end{Problem}

\begin{Problem}
Solve $(x^2-8\cdot x+14)^{1/2}=(6-2\cdot x)^{1/2}$. Show your work.
\end{Problem}

The final four problems test understanding of basic algebraic derivations.
\begin{Problem}
Derive $(a+b)\cdot (c+d)=ac+ad+bc+bd$ from $u\cdot (v+w)=uv+uw$.
\end{Problem}

\begin{Problem}
Derive $(a+b)^{2}=a^{2}+2ab+b^{2}$ from $u\cdot (v+w)=uv+uw$.
\end{Problem}

\begin{Problem}
Derive $a^2-b^2=(a-b)\cdot (a+b)$ from $u\cdot (v+w)=uv+uw$.
\end{Problem}

\begin{Problem}
Derive $(a+b)^{3}=a^{3}+3ab^2+3a^2b+b^{3}$ from $u\cdot (v+w)=uv+uw$.
\end{Problem}

\section{Basic laws of set theory}

In this section basic identities of set theory will be developed. The idea is to treat sets as a primitive concept and create basic axioms governing identities from which all relevant identities can be derived. Delving into the issue of which point belongs where is not necessary. The main concern is with equality of sets. Logic will be minimized as much as possible.

It is asumed that all objects considered by us belong to a non-empty \textbf{universe} $U$. It is a set (we treat the notion of a set as a primitive concept that does not need to be defined). All sets considered have elements belonging to $U$ (i.e. they are subsets of $U$).

\begin{Definition}\label{MinimalUniverseDef}
By a \textbf{minimal universe} $MU$ we mean a set consisting of just one element. $MU=\{x\}$
means $MU$ is a minimal universe and $x$ is its unique element.\\
The \textbf{standard model for a minimal universe} is $SMU=\{\emptyset\}$, the set whose only element is the empty set $\emptyset$ (the set that contains no elements).\\
Sometimes we will employ algebraic notation and declare $\emptyset=0$ in which case $SMU=\{0\}$ will be denoted by $1$.
\end{Definition}

Given two sets $A$, $B$ one considers the following sets:\\
a. Their \textbf{intersection} $A\cap B$ consisting of all elements that belong to both of these sets,\\
b. Their \textbf{union} $A\cup B$ consisting of all elements that belong to at least one of these sets,\\
c. Their \textbf{difference} $A\setminus B$ consisting of all elements that belong to $A$ but not to $B$.\\

The following are basic axioms for operators $\cap$, $\cup$, and $\setminus$:

\begin{Axioms}[Basic Axioms of Set Theory]\label{BasicAxioms}
For any minimal universe $MU$ one has
\par\noindent
1. $MU\cup MU=MU\cap MU=MU$.\\
2. $\emptyset\cup MU=MU\cup\emptyset=MU$.\\
3. $\emptyset\cap MU=MU\cap\emptyset=\emptyset$.\\
4. $MU\setminus MU= \emptyset $.\\
5. $MU\setminus\emptyset=MU$.\\
6. $\emptyset\setminus MU=\emptyset$.\\
\end{Axioms}

Axioms \ref{BasicAxioms} may be used to create new identities. Here is an example: $A\cup (SMU\setminus A)=SMU$ is an identity in the universe $SMU$. Indeed, there are only two possibilities for $A$: $A=\emptyset$ or $A=SMU$. Students should be able to consider both of them and see that the new identity is valid.

The next law provides a mechanism for creation of identities in any universe $U$. We simply reduce a conjectured identity to $SMU$ and verify it there. Example: $A\cup (U\setminus A)=U$ translates to $A\cup (SMU\setminus A)=SMU$ and this one was already verified. The recipe for translation is simple: replace all ocurrences of $U$ by $SMU$.

\begin{Law}[Identity Generating Machine]\label{LawGenerator}
Any equality of sets expressed in terms of union, intersection, and difference is valid if it translates to an identity in the standard minimal universe $SMU$.
\end{Law}
\textbf{Justification:} Suppose an element $x$ belongs to only one side of conjectured equality.
Since $x$ does not interact with any other elements, we may remove all of them.
The result is an equality where all non-empty sets are equal to $\{x\}$. Now we can replace $\{x\}$ by $SMU$ in all occurences and the rest of sets are replaced by the empty set.
$\square$

The Justification is the only place in this section where the issue of points belonging to sets is raised.

\begin{Example}
Let's see how the justification works in the case of equality $A\cap B=B\cap A$. If there is a point $x$ belonging only to the left side, notice all other points do not interact with $x$, so their participation is irrelevant and we may remove them from both $A$ and $B$. That results in converting each of $A$, $B$ to either the empty set or a set equal to $\{x\}$. Since labels are irrelevant, one may simply consider $\{x\}=SMU$ and thus reduce all equalities of the form $A\cap B=B\cap A$ to only those where $A$ and $B$ live in $SMU$. If that equality is verified for the universe $SMU$, then existence of $x$ belonging only to the left side of $A\cap B=B\cap A$ is not possible.
\end{Example}

\begin{Corollary}\label{IntermediateStep} For every subset $A$ of a universe $U$ the following holds:\\
\noindent
1. $A\cup A=A\cap A=A$.\\
2. $\emptyset\cup A=A\cup\emptyset=A$.\\
3. $\emptyset\cap A=A\cap\emptyset=\emptyset$.\\
4. $A\setminus A= \emptyset $.\\
5. $A\setminus\emptyset=A$.\\
6. $\emptyset\setminus A=\emptyset$.\\
7. $A\cup U=U$.\\
8. $A\cap U=A$.
\end{Corollary}

Corollary \ref{IntermediateStep} is the intermediate step before deriving main identities of basic set theory. One uses Axioms and \ref{LawGenerator} to derive \ref{IntermediateStep}.

The following is an easy consequence of \ref{LawGenerator} and is easier to use in practice. 
\begin{Theorem}[Law of Extremes]\label{LawOfExtremes}
Any equality of sets expressed in terms of union, interesection, and difference is valid if it holds in all extreme cases: each set being either empty or
 equal to the current universe $U$.
\end{Theorem}

 The next step is to use \ref{IntermediateStep} and the Law of Extremes to prove propositions below.

\begin{Proposition}[Commutativity of the Union]
$A\cup B=B\cup A$ for all sets $A,B$.
\end{Proposition}

\begin{Proposition}[Commutativity of the Intersection]
$A \cap B=B \cap A$ for all sets $A,B$.
\end{Proposition}

\begin{Proposition}[Distributive Property for the Union]
$A\cup (B\cap C)=(A\cup
B)\cap (A\cup C)$ for all sets $A,B,C$.
\end{Proposition}
\begin{proof}
All possibilities for $A,B,C$ need not be tested. That would amount to using truth tables, a dreadful way of teaching \lq proofs\rq . Notice that $A$ plays a crucial role here. What happens if $A=\emptyset$? What about $A=U$?
\end{proof}

\begin{Proposition}[Distributive Property for the Intersection]
$A\cap (B\cup C)=(A\cap B)\cup (A\cap C)$ for all sets $A,B,C$.
\end{Proposition}

\begin{Proposition}[De Morgan Laws]
$A\setminus (B\cup C)=(A \setminus B)\cap (A \setminus C)$ 
and $A\setminus (B\cap C)=(A \setminus B)\cup (A \setminus C)$ for all sets $A,B,C$.
\end{Proposition}

\section{Complement and symmetric difference}

This section is devoted to two useful operators on sets and its purpose is to allow students to practice development of identities based on basic identities and the Law of Extremes. Also, an evolution of the Law of Extremes with inclusion of more operators can be seen.
\begin{Definition}\label{SymmetricDifferenceDef}
The \textbf{symmetric difference} $A\Delta B$ of two sets consists of all elements that belong to exactly one of these sets. That means $A\Delta B:=(A\setminus B)\cup (B\setminus A)$.
\end{Definition}

\begin{Proposition} For all sets $A$ one has\\
1. $A\Delta A=\emptyset$.\\
2. $\emptyset\Delta A=A\Delta\emptyset=A$.
\end{Proposition}

\begin{Theorem}\label{LawOfExtremesTwo}
Any equality of sets expressed in terms of union, interesection, difference, and symmetric difference is valid if it holds in all extreme cases: each set being either empty or
 equal to the universe $U$.
\end{Theorem}

\begin{Proposition}[Associativity of the Symmetric Difference]
$A\Delta (B\Delta C)=(A\Delta B)\Delta C$ for all sets $A,B,C$.
\end{Proposition}

\begin{Definition}\label{ComplementDef}
Within a given universe $U$ the \textbf{complement} $A^c$ of its subset $A$ is defined as $U\setminus A$.
\end{Definition}

\begin{Proposition}
$\emptyset^c=U$ and $U^c=\emptyset$.
\end{Proposition}

\begin{Proposition}
$(A^c)^c=A$ for any subset $A$ of the universe $U$.
\end{Proposition}

\begin{Proposition}
$(A\cup B)^c=A^c\cap B^c$ and $(A\cap B)^c=A^c\cup B^c$ for any subsets $A$ and $B$ of the universe $U$.
\end{Proposition}

\begin{Theorem}\label{LawOfExtremesThree}
Any equality of sets expressed in terms of union, interesection, difference, symmetric difference, and complements is valid if it holds in all extreme cases: each set being either empty or
 equal to the universe $U$.
\end{Theorem}

Notice that in the equality $A^c=U\Delta A$, $U$ is fixed, so to check it one only needs to consider $A=\emptyset$ or $A=U$.

\begin{Problem}\cite{VINK} 
Prove $A\setminus(A\setminus B)=A\cap B$ for any sets $A,B$.
\end{Problem}

\begin{Problem}\cite{VINK} 
Prove $A\cap(B\setminus C)=(A\cap B)\setminus (A\cap C)$ for any sets $A,B,C$.
\end{Problem}

\begin{Problem}\cite{VINK} 
Prove $A\Delta B)\cap C)=(A\cap C)\Delta (B\cap C)$ for any sets $A,B,C$.
\end{Problem}

\begin{Problem}\cite{VINK} 
Does $A\Delta B)\cup C)=(A\cup C)\Delta (B\cup C)$ hold for any sets $A,B,C$?
\end{Problem}

Notice that at this moment we can introduce students to Abelian groups. Traditionally, groups are introduced in a complicated way. However, my algebraic approach to set theory allows for discussion of Abelian groups at this stage (as an optional assignment).

\begin{Problem}
Show that $2^{SMU}=\{0,1\}$ equipped with the symmetric difference is an Abelian group. 
\end{Problem}

\section{Subsets}

Up to now the concept of subset has been treated intuitively. It is time to define it rigorously. The issue of elements of sets is avoided and we rely on equalities exclusively.

So far proofs amounted to finding creative ways to reduce (via substitution) an equality to another equality previously established. In this section students are asked to transform a known equality to another one using operations $\cap$, $\cup$, ${}^c$, etc. It is at this moment that we are using for the first time essential facts about those operators: equality $X=Y$ induces equalities $X\cap Z=Y\cap Z$, $X\cup Z=Y\cup Z$, $X^c=Y^c$. The issue of implication arises only in the algebraic contex of transforming equalities. Notice that many students understand implication as exactly a way to transform equalities. The full scope of implications will be discussed in the next section.

\begin{Definition}
We say $A$ is a \textbf{subset} of $B$ (notation: $A\subset B$ or $A\subseteq B$) if $A\cup B=B$.
\end{Definition}

\begin{Proposition}
$\emptyset \subset A$ for all sets $A$.
\end{Proposition}

\begin{Proposition}
$A\cap B\subset A$ for all sets $A,B$.
\end{Proposition}

\begin{Proposition} Suppose $A,B$ are subsets of a universe $U$. \\
a. Derive $A\cap B=A$ from $A\cup B=B$.\\
b. Derive $A\cup B=A$ from $A\cap B=A$.\\
c. Conclude $A\subset B$ may be defined as $A\cap B=A$.
\end{Proposition}

\begin{Proposition}
If $A\subset B$ and $B\subset C$, then $A\subset C$.
\end{Proposition}

\begin{Proposition}
If $A\subset B$, then $A\cap C\subset B\cap C$ for all sets $A,B,C$.
\end{Proposition}

\begin{Proposition}
If $A\subset B$, then $B^c\subset A^c$.
\end{Proposition}

As can be seen now, our discussion of a universe $U$ was really about the set of all subsets of $U$.
\begin{Definition}
Given a universe $U$, $2^U$ is the set of all subsets of $U$. It is a natural \textbf{expansion} of our universe $U$. 
\end{Definition}

Notice $SMU=2^{\emptyset}=\{0\}$. Discussing empty set requires at least $SMU$. \lq If a tree falls in a forest and no one is around to hear it, does it make a sound?\rq\ has an answer in set theory.

Another observation is that $U$ can be considered as a subset of $2^U$, namely the set of all minimal sub-universes of $U$.

\section{Logic}

So far the concern has been only with one issue: whether equality of two sets holds or not (with the emphasis on developing a calculus of equalities that always hold). Our basic example of an equality that does not hold is $SMU=\emptyset$. To consider inequalities of sets basic concepts of logic need to be developed.

The basic issue is to recognize whether a given element $x$ belongs to a given set $A$ or not. If it does, it can be expressed by saying $x\in A$ is \textbf{true}. If $x$ does not belong to $A$, then we declare $x\in A$ to be \textbf{false}. This section deals with statements that can be classified as true or false. Our basic statement in this category is $x\in A$, as most statements in undergraduate mathematics are of that form.

Notice \lq $A\cap B=B\cap A$ for all sets $A,B$\rq \ is of a different nature: it has to do with the fact that there is no set of all sets. Thus, it is not built from basic statements of the type $x\in A$. On the other hand, \lq $A\cap B=B\cap A$ for all subsets $A,B$ of a given universe $U$\rq\ is built from basic statements. To see it, one needs to go to the higher universe $2^U$ (the set of all subsets of $U$). Now, consider the set $S$ of all elements $X$ of $2^U$ such that $X\cap Y=Y\cap X$ for all elements $Y\in 2^U$. The statement \lq $A\cap B=B\cap A$ for all subsets $A,B$ of a given universe $U$\rq\ means the set $S$ equals the whole $2^U$.

Typically, Russell's Paradox (see \ref{RussellParadox} below) is formulated as "there is no set of all sets". I will give a different formulation.

\begin{Theorem}[Russell's Paradox]\label{RussellParadox}
A universe $U$ such that $2^U$ is a subset of $U$ leads to breakdown of logic.
\end{Theorem}
\begin{proof}
If $2^U$ is a subset of $U$, then we may consider $A=\{x\in U | x\notin x\}$: the set of all elements of $U$ not containing itself. xxx more details here. Since $A\in 2^U\subset U$, the question is
whether $A\in A$ or not. So put $x=A$ and test $x\notin x$. If $x\in A$, it means
$x\notin x$, i.e. $x\notin A$. If $x\notin A$, it means $x\notin x$, i.e. $x\in A$.

That means any universe $U$ for which $2^U\subset U$ leads to breakdown of logic: we cannot assign the labels true, false to basic statements under this scenario.
\end{proof}

\begin{Remark}
The proof of \ref{RussellParadox} points to the importance of different elements of sets not to interact with each other (see the Justification of the Identity Generating Machine). If $2^U$ is a subset of $U$ then there is a possibility of self-interaction $x\in x$ for elements of $U$.\\
Another way to interpret Russell's Paradox is to say that the process of expanding universes never ends.
\end{Remark}

\begin{Exercise}
Show that for any universe $U$ there is $A\subset U$ such that $A\notin U$.
\end{Exercise}

Russell's Paradox (Theorem \ref{RussellParadox}) points to the fact one has to proceed gingerly in discussing logic. G\" odel's theorem is further evidence: there are statements in arithmetics whose validity cannot be determined for any given finite set of axioms.

Despite those difficulties, we will try to develop a calculus of statements with statements related to set theory (statements of the type $x\in A$) being our basic model.

From now on only with statements that can be deemed either true or false will be dealt with.

\subsection{Alternative}

What is the meaning of $x\in A\cup B$? It means $x$ belongs to $A$ \textbf{or} $x$ belongs to $B$. Thus, $x\in A\cup B$ is false exactly when both $x\in A$ and $x\in B$ are false. That leads to the \textbf{logical alternative operator} $\lor$: given any two statements $p,q$ a new statement $p\lor q$ can be formed. The only time $p\lor q$ can be false is if both $p$ and $q$ are false.

\subsection{Conjunction}

What is the meaning of $x\in A\cap B$? It means $x$ belongs to $A$ \textbf{and} $x$ belongs to $B$. Thus, $x\in A\cap B$ is true exactly when both $x\in A$ and $x\in B$ are true. That leads to the \textbf{logical conjunction operator} $\land$: given any two statements $p,q$ a new statement $p\land q$ can be formed. The only time $p\land q$ can be true is if both $p$ and $q$ are true.

\subsection{Negation}

What is the meaning of $x\notin A$? It means $x$ does \textbf{not} belong to $A$. Thus, $x\notin A$ is true exactly when $x\in A$ is false. That leads to the \textbf{logical negation operator} $\lnot$: given any statement $p$ a new statement $\lnot p$ can be formed. The only time $\lnot p$ can be true is if $p$ is false.

\subsection{Implication (the conditional in \cite{Dev})}

What is the meaning of $A\subset B$? It means $x$ belongs to $B$ \textbf{if} $x$ belongs to $A$. Thus, $A\subset B$ is false exactly when there is $x$ such that $x\in A$ is true and $x\in B$ is false. That leads to the \textbf{logical implication operator} $\implies$: given any two statements $p,q$ a new statement $p\implies q$ can be formed. The only time $p\implies q$ can be false is if $p$ is true and $q$ is false.

Implication causes the most misunderstanding among students. And here is where interpreting it in terms of set theory helps. The most confusing fact is that anything can be implied by falsehood: the meaning of it in set theory is $\emptyset\subset A$ for all sets $A$, which is quite obvious from our definition of being a subset.

\subsection{Equivalence (the biconditional in \cite{Dev})}

What is the meaning of $A=B$? It means $x$ belongs to $B$ \textbf{if and only if} $x$ belongs to $A$. Thus, $A=B$ is false exactly when there is $x$ such that logical values of $x\in A$ and $x\in B$ are different. That leads to the \textbf{logical equivalence operator} $\equiv$: given any two sentences $p,q$ a new statement $p\equiv q$ can be formed. The only time $p\equiv q$ can be false is if $p$ and $q$ have different logical values.

\subsection{Tautologies}

Suppose we see a composite statement $\phi$ out of statements $p$, $q$, $r$, etc. If the new statement $\phi$ is always true regardless of logical values of components $p$, $q$, $r$, $\ldots$, then $\phi$ is called a \textbf{tautology}. The most known tautology is $p\lor \lnot p$ (to be or not to be, in the words of Hamlet).

\section{Interaction between set theory and logic}
This section is devoted to the complete analogy between basic set theory and logic. We are only interested in set-theoretic identities and logical tautologies.

\subsection{From sets to logic}
Consider any set-theoretic expression formed using only operators $\cup$, $\cap$, and $\setminus$ ($(A\cup B)\cap C^c$, for example). What does it mean $x\in (A\cup B)\cap C^c$? Let's see: it means ($x\in A$ or $x\in B$) and $x\notin C$.

Now one sees
\begin{Theorem}
[Recipe for Converting Set-theoretic Equalities]
Given a set-theoretic equality do the following:\\
a. Replace $\cup$ by $\lor$, $\cap$ by $\land$, and ${}^c$ by $\lnot$,\\
b. Replace each set $Y$ by the statement \lq $x\in Y$\rq\ with $x$ being the same in all replacements.

If one arrives at a tautology, then the original equality is an identity (valid for all sets appearing there).
\end{Theorem}

How does one check if the logical expression is a tautology? By deriving it from other tautologies or by applying the truth tables in the worst case. 
Now one arrives at
\begin{Theorem}
[Simplified Recipe for Converting Set-theoretic Equalities]
Given a set-theoretic equality do the following:\\
a. Replace $\cup$ by $\lor$, $\cap$ by $\land$, and ${}^c$ by $\lnot$,\\
b. Replace each set $Y$ by the statement \lq $Y$ is not empty\rq\ which is simplified to $\bar Y$.

If one arrives at a tautology, then the original equality is an identity (valid for all sets appearing there).
\end{Theorem}

 At this moment one can add inclusion $A\subset B$ and equality $A=B$ of sets by noticing they translate to $\bar A\implies \bar B$ and $\bar A\equiv \bar B$, respectively.

\subsection{From logic to sets}

What about the other direction? It turns out basic logic can be interpreted in terms of subsets of $SMU$. Given a statement $p$, there are only two issues: is $p$ true or is it false. Given a subset $A$ of $SMU$, there are only two issues: is $A$ empty or is it equal $SMU$. Thus one may transfer each statement $p$ to a subset $\bar p$ of $SMU$ according to the following rule: $\bar p=\emptyset$ if $p$ is false, and $\bar p=SMU$ if $p$ is true.

From that point of view $p\lor q$ translates to $\bar p\cup \bar q$, $p\land q$ translates to $\bar p\cap \bar q$, and $\lnot p$ translates to $\bar p^c$. Moreover, $p\equiv q$ translates to $\bar p=\bar q$ and $p\implies q$ translates to $\bar p\subset \bar q$.

Since $\bar p\cup \bar p^c=SMU$ holds, it means $p\lor \lnot p$ is always true (it is a tautology).

Suppose one wants to use set-theoretic identities to verify tautologies. One needs to figure out how to transform logical identities into set-theoretical identities. The only remaining operators to consider are equivalence $\equiv$ and implication $\implies$. Notice $p\implies q$ is equivalent to $\lnot p\lor q$. This means, whenever we see implication $p\implies q$, one should transform it to $\bar p^c\cup \bar q$.

\begin{Example}
Suppose we want to show $(p\implies q)\equiv (\lnot q\implies \lnot p)$.
\end{Example}
In set theory it means $\bar p^c\cup \bar q=(\bar q^c)^c\cup \bar p^c$, which is indeed true.

\begin{Problem}
Suppose $p$ and $q$ are statements. Prove that 
$(p\land q)\implies p$ is a tautology.
\end{Problem}

\begin{Problem}
Suppose $p$ and $q$ are statements. Prove that 
$p \implies (p\lor q)$ is a tautology.
\end{Problem}

\begin{Problem}
Suppose $p$ is a statement. Prove that $(\lnot p \implies  p) \implies p$ is a tautology .
\end{Problem}

\begin{Problem}
Suppose $p$ and $q$ are statements. Prove that $(\lnot p \implies  q\land\lnot q) \implies p$ is a tautology.
\end{Problem}

\begin{Problem}
Suppose $p$ and $q$ are statements. Prove that $(q\land (\lnot p \implies  \lnot q)) \implies p$ is a tautology.
\end{Problem}

\begin{Problem}
Suppose $p$ and $q$ are statements. Prove that $p\equiv q$ if and only if $\lnot p\equiv\lnot q$.
\end{Problem}

\begin{Problem}
Suppose $p$ and $q$ are statements. Prove that $\lnot(p\lor q)$ if and only if $\lnot p\land\lnot q$.
\end{Problem}

\begin{Problem}
Suppose $p$ and $q$ are statements. Prove that $\lnot(p\land q)$ if and only if $\lnot p\lor\lnot q$.
\end{Problem}

\begin{Problem}
Suppose $p$ and $q$ are statements. Prove that $p\Rightarrow q$ if and only if $\lnot(p\land\lnot q)$.
\end{Problem}

\begin{Problem}
Suppose $p$ and $q$ are statements. Prove that $p\Rightarrow q$ if and only if $\lnot q\Rightarrow \lnot p$.
\end{Problem}

\begin{Problem}
Suppose $p$, $q$, and $r$ are statements. Prove that $p\Rightarrow (q\lor r)$ if and only if $(p\land\lnot q)\Rightarrow r$.
\end{Problem}

\begin{Problem}
Suppose $p$, $q$, and $r$ are statements. Prove that $p\Rightarrow (q\land r)$ if and only if $(p\Rightarrow r)\land (p\Rightarrow q)$.
\end{Problem}

\begin{Problem}
Suppose $p$, $q$, and $r$ are statements. Prove that $(p\lor q) \Rightarrow r$ if and only if $(p \Rightarrow r)\land (q \Rightarrow r)$.
\end{Problem}

\begin{Problem}
Suppose $p$, $q$, and $r$ are statements. Prove that 
$(p\lor q)\land r\equiv (p\land r)\lor (q\land r)$.
\end{Problem}

\begin{Problem}
Suppose $p$, $q$, and $r$ are statements. Prove that 
$(p\land q)\lor r\equiv (p\lor r)\land (q\lor r)$.
\end{Problem}

\section{Transformations of sets}
It is customary to consider transformations of sets to be a more advanced concept and to start with the concept of a function. In my view transformations are intrinsically connected to the concept of a set. Also, the material below allows students to use their algebraic skills and revisit creation of set-theoretic identities.

Notice every universe $U$ carries with it the beginning of the concept of a function. Namely, given a set $A$ one has three possible transformations of any set $X$:\\
1. $X\to A\cap X$,\\
2. $X\to A\cup X$,\\
3. $X\to X^c$.

Thus, given any set $X$ we can produce a new set or we can transform it into a new set. Let's use the following notations:
\begin{Definition}
Given a subset $A$ of a universe $U$ let the \textbf{transformation} $i_A$ be defined by
$$i_A(X):=A\cap X$$
Similarly, $u_A$ is defined by
$$u_A(X):=A\cup X$$
and $c$ is defined by
$$c(X):=X^c$$
\end{Definition}
Notice there is an additional transformation, the \textbf{identity} $id$, defined by
$$id(X):=X$$

\begin{Definition}
Given any two transformations $t$ and $s$ of subsets of $U$, we define their \textbf{composition} $t\circ s$ by
$$t\circ s(X)=t(s(X))$$
$t^2$ is a shortcut for $t\circ t$.
\end{Definition}

\begin{Question}
What is $i_A^2$?
\end{Question}

\begin{Question}
What is $u_A^2$?
\end{Question}

\begin{Question}
What is $c^2$?
\end{Question}

\begin{Question}
What is $i_A\circ i_B$?
\end{Question}

\begin{Question}
What is $u_A\circ u_B$?
\end{Question}

\begin{Problem}
Show $i_A(X\cup Y)=i_A(X)\cup i_A(Y)$.
\end{Problem}

\begin{Problem}
Show $i_A(X\cap Y)=i_A(X)\cap i_A(Y)$.
\end{Problem}

\begin{Problem}
Show $u_A(X\cup Y)=u_A(X)\cup u_A(Y)$.
\end{Problem}

\begin{Problem}
Show $c(X\cup Y)=c(X)\cap c(Y)$.
\end{Problem}

\begin{Problem}
Show $c(X\cap Y)=c(X)\cup c(Y)$.
\end{Problem}

\begin{Problem}
Show $i_A=c\circ u_{A^c}\circ c$.
\end{Problem}

\begin{Problem}
Show $u_A=c\circ i_{A^c}\circ c$.
\end{Problem}

Notice the Axiom of Choice can be reformulated as follows (see \ref{ACDiff}):
\begin{Axiom}[Axiom of Choice]\label{AxiomOfChoice}
For every universe $U$ there is a transformation $\phi:2^U\to 2^U$ such that $\phi(A)\subset A$ and $\phi(A)$ is a minimal universe for each non-empty $A\subset U$.
\end{Axiom}

This formulation is more adequate to the spirit of this paper. The transformation $\phi$ in \ref{AxiomOfChoice} may be thought of as a kind of a feedback mechanism between $U$ and its expansion $2^U$. From that point of view \ref{AxiomOfChoice} makes perfect sense (sorry, Axiom Of Choice deniers).

\section{Cartesian product}

Cartesian product provides a way for different elements to interact. This culminates in the concept of a relation.

To define cartesian product one needs the concept of an ordered pair $(a,b)$. Notice $2^U$ comes with a natural partial order. Namely, $A\leq B$ can be defined as $A\subset B$. It makes perfect sense to use that partial order to define ordered pairs. From the point of view of $2^U$, an ordered pair should be a pair of different subsets $A$ and $B$ of $U$ such that $A\subset B$ or $B\subset A$. The smaller set is equal to $A\cap B$ and the larger set is equal to $A\cup B$.

\begin{Definition}
The \textbf{ordered pair} $(a,b)$ is the set $\{\{a\},\{a,b\}\}$.\\
The \textbf{cartesian product} $X\times Y$ is the set of all ordered pairs $(a,b)$, where $a\in X$ and $b\in Y$. Thus, $X\times Y$ is a subset of the double expansion $2^{2^U}$ of the universe $U$.
\end{Definition}
\begin{Remark}
Notice other authors arive at the same definition of $(a,b)$ differently (see \cite{Hal}, p.23).
\end{Remark}

In my teaching of graduate courses in topology I have seen students who thought
$X\times Y\setminus A\times B=(X\setminus A)\times (Y\setminus B)$ holds (it has to do with a problem in \cite{Mun}  on connectedness of $ X\times Y\setminus A\times B$ if $A$ and $B$ are proper subsets of connected sets $X$ and $Y$). Notice they would benefit from the idea of checking it for basic sets (what happens if $A=\emptyset\ne B$?).

The correct formula is
$$(X\times Y)\setminus (A\times B)=((X\setminus A)\times Y)\cup (X\times (Y\setminus B))$$ yet it is insufficient to test it only in extreme cases. Indeed, the cartesian product creates sets outside of a given universe $U$.

To prove $(X\times Y)\setminus (A\times B)=((X\setminus A)\times Y)\cup (X\times (Y\setminus B))$ one needs to observe that the statement \lq $C\times D$ is not empty\rq\ is equivalent to $\bar C\land \bar D$. 
Now the left side translates to $(\bar X\land \bar Y)\land \lnot(\bar A\land \bar B)$
and the right side translates to 
$(\bar X\land \lnot \bar A\land \bar Y)\lor (\bar X\land \bar Y\land \lnot \bar B)$.

In this case it is more beneficial to go from set theory to logic and resolve the issue there.

What happens if one goes from $(\bar X\land \bar Y)\land \lnot(\bar A\land \bar B)\equiv (\bar X\land \lnot \bar A\land \bar Y)\lor (\bar X\land \bar Y\land \lnot \bar B)$ back to set theory?
One gets the identity
$$(X\cap Y)\setminus (A\cap B)=((X\setminus A)\cap Y)\cup (X\cap (Y\setminus B))$$

Here are a few more pairs of related identities:\\
1. $X\times (A\cup B)=X\times A\cup X\times B$ and $X\cap (A\cup B)=X\cap A\cup X\cap B$,\\
2. $X\times (A\cap B)=X\times A\cap X\times B$ and $X\cap (A\cap B)=(X\cap A)\cap (X\cap B)$.

Each pair transfers to the same logical tautology.

Can we make it into a fancy theorem? Yes, we can!

\begin{Theorem}\label{ProductVsIntersection}
Suppose there is an equality $E$ involving cartesian products. If the 1st coordinate is independent of the second coordinate, then the validity of $E$ reduces to the validity of the equality obtained from $E$ by replacing $\times$ with $\cap$.
\end{Theorem}

Here is an example of an equality where the 1st coordinate is not independent of the second coordinate:
$$X\times X\setminus A\times A=(X\setminus A)\times (X\setminus A)$$
(the set $A$ appears in both coordinates). The equality is false in general, yet the corresponding equality
$$X\cap X\setminus A\cap A=(X\setminus A)\cap (X\setminus A)$$
always holds.

\begin{Remark}
Notice \ref{ProductVsIntersection} provides a context in which we may use Law of Extremes again. That shows its great utility.
\end{Remark}

\section{Functions}

\begin{Definition}
Given two sets $X$ and $Y$, a \textbf{function} $f:X\to Y$ is a subset $\Gamma(f)$ (the \textbf{graph} of $f$) of $X\times Y$ with the property that for each $x\in X$ there is unique $y\in Y$ (denoted by $f(x)$) so that $(x,y)\in \Gamma(f)$.
\end{Definition}

\begin{Definition}
Any function $f:X\to Y$ leads to two transformations: $f:2^X\to 2^Y$ and $f^{-1}:2^Y\to 2^X$:\\
1. $y\in f(A) \iff y=f(x)$ for some $x\in A$.\\
2. $x\in f^{-1}(A) \iff f(x)\in A$.
\end{Definition}

What is of interest is under what conditions do transformations $f$ and $f^{-1}$ preserve set-theoretic identities.

\begin{Problem}
Characterize functions $f:X\to Y$ such that $f(A\cup B)=f(A)\cup f(B)$ for all $A,B\in 2^X$.
\end{Problem}

\begin{Problem}
Characterize functions $f:X\to Y$ such that $f(A\cap B)=f(A)\cap f(B)$ for all $A,B\in 2^X$.
\end{Problem}

\begin{Problem}
Characterize functions $f:X\to Y$ such that $f(A^c)=f(A)^c$ for all $A\in 2^X$.
\end{Problem}

\begin{Problem}
Characterize functions $f:X\to Y$ such that $f(A\setminus B)=f(A)\setminus f(B)$ for all $A,B\in 2^X$.
\end{Problem}

\begin{Problem}
Characterize functions $f:X\to Y$ such that $f^{-1}(A\cup B)=f^{-1}(A)\cup f^{-1}(B)$ for all $A,B\in 2^Y$.
\end{Problem}

\begin{Problem}
Characterize functions $f:X\to Y$ such that $f^{-1}(A\cap B)=f^{-1}(A)\cap f^{-1}(B)$ for all $A,B\in 2^Y$.
\end{Problem}

\begin{Problem}
Characterize functions $f:X\to Y$ such that $f^{-1}(A^c)=f^{-1}(A)^c$ for all $A\in 2^Y$.
\end{Problem}

\begin{Problem}
Characterize functions $f:X\to Y$ such that $f^{-1}(A\setminus B)=f^{-1}(A)\setminus f^{-1}(B)$ for all $A,B\in 2^Y$.
\end{Problem}

\begin{Problem}
Suppose $f:X\to Y$ is a function and $B\subset Y$. Show $f(f^{-1}(B))=B\cap f(X)$.
\end{Problem}

The following two problems are variants of Russell's Paradox.

\begin{Problem}
Suppose $f_i:\mathbb N\to \{0,1\}$ is a sequence of functions from the set of natural numbers $\mathbb N$. Show there is a function $g:\mathbb N\to \{0,1\}$ such that $g\ne f_i$ for all $i$.
\end{Problem}

\begin{Problem}
Suppose $\phi:2^U\to U$ is a function. Show there existence of two different subsets $A,B$ of $U$ such that $\phi(A)=\phi(B)$.
\end{Problem}

\section{Arbitrary unions and intersections}

Up to now, our attention was focused on finite unions and finite intersections. It is time to extend our basic axiom to arbitrary unions and arbitrary intersections.

The notation $\{A_s\}_{s\in S}$ is a shortcut for a function $A:S\to 2^U$. Thus, $A_s$ is another way of denoting the value $A(s)$ of the function $A$ and $A_s$ is a subset of $U$.

\begin{Definition}
Suppose $\{A_s\}_{s\in S}$ is a family of subsets of a universe $U$. Their \textbf{union}
$\bigcup_{s\in S} A_s$ is the set of all elements of $U$ that belong to at least one of the sets $\{A_s\}_{s\in S}$.\\
Their \textbf{intersection}
$\bigcap_{s\in S} A_s$ is the set of all elements of $U$ that belong to each of the sets $\{A_s\}_{s\in S}$.
\end{Definition}

By checking the Justification of the Identity Generating Machine \ref{LawGenerator} we see that it can be extended as follows:
\begin{Law}[Extended Law of Extremes]\label{ExtendedLawOfExtremes}
Any equality of sets expressed in terms of arbitrary unions, arbitrary interesections, and difference is valid if it holds in all extreme cases: each set being either empty or
 equal to the current universe $U$.
\end{Law}

\begin{Problem}
Show $B\cap (\bigcup_{s\in S} A_s)= \bigcup_{s\in S} (A_s\cap B)$.
\end{Problem}

\begin{Problem}
Show $B\cup (\bigcup_{s\in S} A_s)= \bigcup_{s\in S} (A_s\cup B)$.
\end{Problem}

\begin{Problem}
Show $ (\bigcap_{s\in S} A_s)\cap (\bigcap_{s\in S} B_s)= \bigcap_{s\in S} (A_s\cap B_s)$.
\end{Problem}

\begin{Problem}
Show $(\bigcup_{s\in S} A_s)^c= \bigcap_{s\in S} A_s^c$.
\end{Problem}

Here is a good project for very good students: extend Law of Extremes to arbitrary cartesian products.

\begin{Definition}
Suppose $\{A_s\}_{s\in S}$ is a family of subsets of a universe $U$. Their \textbf{product}
$\prod_{s\in S} A_s$ is the set of all functions $f:S\to U$ such that $f(s)\in A_s$ for each $s\in S$.
\end{Definition}

\begin{Problem}
Suppose $A_i$, $i\ge 1$, are infinite subsets of $U$. Show there exists $B\subset U$ such that $B\ne A_i$ for all $i$. Conclude $2^U$ is not countable.
\end{Problem}

\section{Quantifiers}

Statements of the type \lq there exists\rq\ and \lq for all \rq\ originate from set theory. To say \lq $A$ is not empty\rq\ is the same as saying \lq there exists $x\in A$\rq\ . To say \lq $A$ is a subset of $B$\rq\ is the same as saying \lq for all $x\in A$ one has $x\in B$\rq\ .

How to transfer those concepts to logic?

Suppose $\{p_n\}_{n=1}^\infty$ is an infinite sequence of statements. What does it mean $p_1\land p_2\land\ldots$ is true? What does it mean $p_1\lor p_2\lor\ldots$ is true? 

Using our strategy of converting statements $p$ to subsets $\bar p$ of $SMU$, one sees that $p_1\land p_2\land\ldots$ converts to $\bar p_1\cap \bar p_2\cap\ldots$ and that set is non-empty if and only if all $\bar p_n$ are non-empty. Therefore, the criterion for $p_1\land p_2\land\ldots$ being true is: all $p_n$ must be true.

Similarly, $p_1\lor p_2\lor\ldots$ converts to $\bar p_1\cup \bar p_2\cup\ldots$ and that set is non-empty if and only if at least one $\bar p_n$ is non-empty. Therefore, the criterion for $p_1\lor p_2\lor\ldots$ being true is: at least one $p_n$ must be true.

Now we apply the following shortcuts: $p_1\land p_2\land\ldots$ is simplified
to $\forall_{n\ge 1}\ p_n$ and $p_1\lor p_2\lor\ldots$ is simplified
to $\exists_{n\ge 1}\ p_n$.

Instead of a sequence of statements one may consider arbitrary sets of statements $\{p_s\}_{s\in S}$.

\begin{Definition}
$\forall_{s\in S}\ p_s$ is true exactly when all $\{p_s\}_{s\in S}$ are true.

$\exists_{s\in S}\ p_s$ is true exactly when at least one $\{p_s\}_{s\in S}$ is true.
\end{Definition}

Now the strategy of converting $\forall_{s\in S}\ p_s$ to $\bigcap_{s\in S} \bar p_s$ and converting $\exists_{s\in S}\ p_s$ to $\bigcup_{s\in S} \bar p_s$ works, so one can use it to derive new tautologies of logic.

\begin{Problem}
Show $\forall_{s\in S}\ p_s\implies \exists_{s\in S}\ p_s$.
\end{Problem}

\begin{Problem}
Show $\lnot(\forall_{s\in S}\ p_s)\equiv \exists_{s\in S}\ \lnot p_s$.
\end{Problem}

\begin{Problem}
Show $\forall_{s\in S}\ (p\implies q_s)\equiv (p\implies (\forall_{s\in S}\ q_s))$.
\end{Problem}

\begin{Problem}
Show $(\forall_{s\in S}\ (p_s\implies q_s))\implies ((\forall_{s\in S}\ p_s)\implies (\forall_{s\in S}\ q_s))$.
\end{Problem}

\begin{Problem}
Show $\exists_{t\in T}(\forall_{s\in S}\ p_{s,t})\implies (\forall_{s\in S} \exists_{t\in T}\ p_{s,t})$.
\end{Problem}

\section{Equivalence relations}

The concept of an equivalence relation $\sim$ on a set $X$ and the concept of the set $X/{\sim}$ of equivalence classes are quite difficult for students.

My view is that the best way to approach them is from the expansion $2^X$ of $X$.

Most abstract constructions in mathematics involve creating a partition of a given set $X$.
\begin{Definition}
A \textbf{partition} of a set $X$ is a subset $\mathcal{P} $ of $2^X$ with the property that
$\bigcup \mathcal{P} =X$ and $A\cap B=\emptyset$ for every two different elements $A$ and $B$ of $\mathcal{P} $.
\end{Definition}

\begin{Example}
One can partition the set of integers $\mathbb Z$ into two sets: $Even$  consisting of all even integers and  $Odd$ consisting of all odd integers. Notice that declaring $Even+Even=Even$, $Even+Odd=Odd$, and $Odd+Odd=Even$ turns this partition into an Abelian group.
\end{Example}

To describe a partition $\mathcal{P} $ intrinsically (from the point of view of $X$) one needs to give necessary and sufficient conditions for two points $x,y\in X$ to belong to the same element of $\mathcal{P} $. That leads to the concept of a relation on the set $X$.

\begin{Definition}
A \textbf{relation} $R$ on a set $X$ is a subset of $X\times X$. A common shortcut for $(x,y)\in \mathcal{R} $ is $x\mathcal{R} y$.
\end{Definition}

\begin{Example}
Each partition $\mathcal{P} $ of $X$ induces a natural relation $x\sim_{\mathcal{P} } y$ on $X$. Namely,
$x\sim_{\mathcal{P} } y$ if and only if $x$ and $y$ belong to the same element $A$ of the partition $\mathcal{P} $.
\end{Example}

The relation $x\sim_{\mathcal{P} } y$ has natural properties:\\
1. $x\sim_{\mathcal{P} } x$ for all $x\in X$;\\
2. $x\sim_{\mathcal{P} } y$ implies $y\sim_{\mathcal{P} } z$;\\
3. $x\sim_{\mathcal{P} } y$ and $y\sim_{\mathcal{P} } z$ imply $x\sim_{\mathcal{P} } z$.

That can be generalized as follows:
\begin{Definition}
A relation $\mathcal{R} $ on $X$ is \textbf{reflexive} if $x\mathcal{R} x$ for all $x\in X$.\\
A relation $\mathcal{R} $ on $X$ is \textbf{symmetric} if $x\mathcal{R} y$ implies $y\mathcal{R} x$ for all $x,y\in X$.\\
A relation $\mathcal{R} $ on $X$ is \textbf{transitive} if $x\mathcal{R} y$ and $y\mathcal{R} z$ implies $x\mathcal{R} z$ for all $x,y,z\in X$.\\
$\mathcal{R} $ is an \textbf{equivalence relation} if it is reflexive, symmetric, and transitive.
\end{Definition}

\begin{Problem}
For every equivalence relation $\mathcal{R} $ on $X$ there is a unique partition $\mathcal{P} $ of $X$ such that $\mathcal{R} =\sim_{\mathcal{P} }$.
\end{Problem}

\section{Ordering}

Notice $2^U$ comes with a natural order, namely the inclusion of sets. It is a basic example of a partial order.

\begin{Definition}
A (non-strict) \textbf{partial order} is a binary relation "$\leq $" over a set $P$ which is antisymmetric, transitive, and reflexive, i.e., for all $a$, $b$, and $c$ in $P$, we have that:\\
1. $a \leq  a$ (\textbf{reflexivity});\\
2. if $a \leq  b$ and $b \leq  a$ then $a = b$ (\textbf{antisymmetry});\\
3. if $a \leq  b$ and $b \leq  c$ then $a \leq  c$ (\textbf{transitivity}).
\end{Definition}

A chain of subsets of $U$ is a family $\mathcal{S}$ of subset such that any two elements $A,B\in \mathcal{S}$ are in a relation: $A\subset B$ or $B\subset A$. This leads to natural generalizations.

\begin{Definition}
A \textbf{total order} is a partial order "$\leq $" over a set $P$ such that any two elements $a,b\in P$ are in a relation: $a\leq b$ or $b\leq a$.
\end{Definition}

\begin{Definition}
A \textbf{chain} in a partially ordered set $P$ is a subset $T$ that becomes totally ordered under the
induced order from $P$.
\end{Definition}

Operations of intersection and union lead to natural enlarging of chains of subsets.

\begin{Problem}
Suppose  $\mathcal{S}$ is a chain of subsets of $U$ and $\mathcal{F}\subset \mathcal{S}$.
Show $\mathcal{S}\cup \{\bigcup \mathcal{F}\}$ is a chain.
\end{Problem}

\begin{Problem}
Suppose  $\mathcal{S}$ is a chain of subsets of $U$ and $\mathcal{F}\subset \mathcal{S}$.
Show $\mathcal{S}\cup \{\bigcap \mathcal{F}\}$ is a chain.
\end{Problem}

Of great interest are maximal chains.

\begin{Definition}
A chain $T$ in a partially ordered set $P$ is \textbf{maximal} if $T=T'$ for every chain $T'$ containing $T$.
\end{Definition}

\begin{Problem}
Find the maximal chain of subsets of natural numbers $\mathbb N$ containing sets $A_n=\{1,\ldots,n\}$.
\end{Problem}

\begin{Problem}
Suppose  $\mathcal{M}$ is a maximal chain of subsets of $U$ and $\mathcal{F}\subset \mathcal{M}$.
Show $\bigcup \mathcal{F}\in \mathcal{M}$.
\end{Problem}

\begin{Problem}\label{MinimalElementsOfMaxChains}
Suppose  $\mathcal{M}$ is a maximal chain of subsets of $U$ and $\mathcal{F}\subset \mathcal{M}$.
Show $\bigcap \mathcal{F}\in \mathcal{M}$.
\end{Problem}

Problem \ref{MinimalElementsOfMaxChains} leads naturally to the concept of well-ordering.

The set of natural numbers is a special case of a {\bf well-ordered set}. The most interesting basic proofs:\\
a. $\sqrt{2}$ being irrational;\\
b. Existence of infinitely many primes;\\
c. Uncountability of irrational numbers;\\
all depend on the order of natural numbers.

Our definition of well-order is designed to create interesting assignments for students.

\begin{Definition}\label{WellOrderDef}
$(X,<)$ is a {\bf well-ordered set} if \lq$<$\rq\ is a binary relation 
such that every non-empty subset $A$ of $X$ has a minimum.
That means $a < b$ implies $a\ne b$ and there is a function $\min:2^X\to X$
 such that $\min(A)\in A$ if $A\ne\emptyset$ and
$\min(A) < x$ for all $x\in A\setminus \{\min(A)\}$.
\end{Definition}

\begin{Problem}
Suppose $(X,<)$ is a well-ordered set and $a\ne b$ in $X$.
Show that $a < b$ or $b < a$.
\end{Problem}

\begin{Problem}
Suppose $(X,<)$ is a well-ordered set and $a,b,c\in X$.
Show that $a < b$ and $b < c$ implies $a < c$.
\end{Problem}

\begin{Remark}
Given a well-ordering $<$ on a set $X$ we may \textbf{extend} it to a partial order $\leq$ by declaring $a\leq b$ if and only if $a=b$ or $a < b$. \\
Conversely, given a partial order $\leq$ on $X$ we may define a binary relation $<$ on $X$ be declaring $a < b$ if $a\ne b$ and $a\leq b$ (notice $<$ may not be a well-ordering).\\
We will use this process of toggling between $\leq$ and $<$ quite often.
\end{Remark}

\begin{Definition}
Given a function $f:X\to Y$ from a set $X$ to a partially ordered set $Y$ (respectively, well-ordered set $Y$) we may \textbf{pull-back} the ordering relation from $Y$ to $X$ as follows: $x_1\leq x_2$ if and only if $f(x_1)\leq f(x_2)$ (respectively, $x_1 < x_2$ if and only if $f(x_1) < f(x_2)$).
\end{Definition}

\begin{Problem}
Show that the pull-back of a partial ordering is always a partial ordering.
\end{Problem}

\begin{Problem}
Find necessary and sufficient conditions for the pull-back of a well-ordering to be a well-ordering.
\end{Problem}

\begin{Notation}
Given a well-ordering $<$ on a set $X$ and given $x\in X$ we will use notation $(-\infty,x]$ for the set
$\{y\in X | y\leq x\}$ and $(-\infty,x)$ for the set
$\{y\in X | y < x\}$.
\end{Notation}

As we can see well-ordering on a universe $U$ is related to transformations introduced in \ref{AxiomOfChoice}. It is more convenient to discuss choice (or selection) functions.

\begin{Definition}
A \textbf{choice function} (or \textbf{selection function}) is a function $\phi:2^U\to U$ such that $\phi(A)\in A$ for every $A\ne\emptyset$.
\end{Definition}

The remainder of this section is devoted to the fact that choice functions provide a feedback mechanism from $2^U$ to $U$ resulting in transferring the partial order on $2^U$ into a well-ordering of $U$. This is the most difficult part of the material but should be accessible to talented honors students.

\begin{Problem}
Show that a partial order $\leq$ on $U$ leads to a well-ordering of $U$ if and only if there is a choice function $\phi:2^U\to U$ such that $A\subset B$ implies $\phi(B)\leq \phi(A)$.
\end{Problem}

\begin{Theorem}\label{FromChoiceToWellOrder}
If there is a choice function $\phi:2^{U}\to U$, then $U$ can be well-ordered.
\end{Theorem}
\begin{proof}
Suppose $U$ cannot be well-ordered. That means for every well-ordering on a subset $X$ of $U$
there is $m(X)\in U\setminus X$. Existence of $\phi$ makes it possible for $m$ to be a function. Consider the family $\mathcal{F}$ of all well-ordered subset $(X,<_X)$ of $U$ with initial element $m(\emptyset)$ and
$x=m(-\infty,x)$ for every $x\in X$.

\textbf{Claim 1}. If $(X,<_X), (Y,<_Y)\in F$ and $X\subset Y$, then $X$ is an initial subset of $Y$ and the restriction of order $<_Y$ to $X$ equals $<_X$.\\
\textbf{Proof of Claim 1}: Consider the set $Z$ of all $x\in X$ such that the interval $(-\infty,x]$ in $X$
equals the interval $(-\infty,x]$ in $Y$. If $Z=X$, we are done. Otherwise, consider $z=\min(X\setminus Z)$ ($\min$ being the minimum function with respect to the order $<_X$). Now, the interval $(-\infty,z)$ in $X$
equals the interval $(-\infty,z)$ in $Y$ and $z=m(-\infty,x)$ resulting in the interval $(-\infty,z]$ in $X$
being equal the interval $(-\infty,z]$ in $Y$, a contradiction.

\textbf{Claim 2}. If $(X,<_X), (Y,<_Y)\in F$ and $Y$ is not a subset of $X$, then $X$ is an initial subset of $Y$ and the restriction of order $<_Y$ to $X$ equals $<_X$.\\
\textbf{Proof of Claim 2}: Consider the set $Z$ of all $y\in Y$ such that the interval $(-\infty,y]$ is not a subset of $X$. $Z\ne\emptyset$ as $Y\subset X$ otherwise. Consider $z=\min(Z)$ ($\min$ being the minimum function with respect to the order $<_Y$). Now, the interval $(-\infty,z)$ in $Y$ is a subset of $X$. Either it is equal to $X$ (resulting in $X\subset Y$) or $z=m(-\infty,z)$ belongs to $X$, a contradiction.

Claim 1 and 2 imply that $V=\bigcup{\mathcal{F} }$ is well-ordered. One can enlarge it to a well-ordered set $V\cup\{m(V)\}\in \mathcal{F}$, a contradiction.
\end{proof}

\section{Variants of Axiom of Choice}

This section is devoted to alternative formulations of the Axiom of Choice.

\begin{Problem} \label{ACDiff}
Show that Axiom of Choice \ref{AxiomOfChoice} is equivalent to the following:
$\prod_{s\in S} A_s$ is not empty for each family $\{A_s\}_{s\in S}$ of non-empty sets.
\end{Problem}

\begin{Problem}
Show that Axiom of Choice \ref{AxiomOfChoice} is equivalent to the following:
For every partition $\mathcal{P} $ on a set $X$ there is a choice function $\phi:\mathcal{P} \to X$ such that $\phi(A)\in A$ for all $A\in \mathcal{P} $.
\end{Problem}

\begin{Problem}
Show that Axiom of Choice \ref{AxiomOfChoice} is equivalent to the following:
Given a family $\{A_s\}_{s\in S}$ of mutually disjoint non-empty sets
there is a subset $B$ of $\bigcup\limits_{s\in S} A_s$
such that $B\cap A_s$ is a one-point set for each $s\in S$.
\end{Problem}

\begin{Problem}
Show that Axiom of Choice \ref{AxiomOfChoice} is equivalent to the following:
Given a family $\{A_s\}_{s\in S}$ of non-empty sets
there is a function $f\colon S\to\bigcup\limits_{s\in S} A_s$
such that $f(s)\in A_s$ for each $s\in S$.
\end{Problem}

\begin{Problem}
Show that Axiom of Choice \ref{AxiomOfChoice} is equivalent to the following:
For any surjective function $p:E\to B$ there is a section $s:B\to E$ (that means $p\circ s=id_B$).
\end{Problem}

\section{Alternatives to the Axiom of Choice}

\begin{Definition}
Suppose $P$ is a partially ordered set. A subset $T$ has an \textbf{upper bound} $u$ in $P$ if $t \leq u$ for all t in $T$. Note that $u$ is an element of $P$ but need not be an element of $T$. An element $m$ of $P$ is called a \textbf{maximal element} if there is no element $x$ in $P$ for which $m < x$.

A \textbf{chain} in $P$ is a totally ordered subset of $P$ (in the order induced from $P$).
\end{Definition}

\begin{Problem}
Suppose there is a choice function $\phi:2^U\to U$ and $U$ is partially ordered so that every chain has an upper bound in $U$. Show $U$ has a maximal element.
\end{Problem}

\begin{Problem}
Consider the set $\mathcal{C} $ of all functions $\psi:A\to U$, where $A\subset 2^U$, such that $\psi(B)\in B$ if $B\in A$ is non-empty. Define $\psi_1\leq \psi_2$ if $\psi_1\subset \psi_2$ (recall $\psi:A\to U$ is a subset of $A\times U\subset 2^U\times U$). Show every chain in $\mathcal{C} $ has an upper bound.
\end{Problem}

\begin{Problem}
Consider the set $\mathcal{C} $ of all functions $\psi:A\to U$, where $A\subset 2^U$, such that $\psi(B)\in B$ if $B\in A$ is non-empty. Define $\psi_1\leq \psi_2$ if $\psi_1\subset \psi_2$ (recall $\psi:A\to U$ is a subset of $A\times U\subset 2^U\times U$). Show every maximal element in $\mathcal{C} $ is a choice function.
\end{Problem}

\begin{Problem}
Suppose two chains $C$ and $D$ in a partially ordered set $P$ are well-ordered and $C\cup D$ is a chain. Show $C\cup D$ is well-ordered.
\end{Problem}

\begin{Problem}
Suppose a chain $C$ in a partially ordered set $P$ is well-ordered and is maximal in the family of all well-ordered chains in $P$. Show $C$ is maximal in the family of all chains in $P$.
\end{Problem}

\begin{Definition}
A family $\mathcal{FC}$ of subsets of a set $X$
is of \textbf{finite character} if $\emptyset\in\mathcal{FC}$ and $A\in\mathcal{FC}$ is equivalent to all finite subsets of $A$ belonging to $\mathcal{FC}$.
\end{Definition}

A primary example of a family of finite character is the family of all linearly independent subsets of a vector space.

\begin{Problem}
Suppose $P$ is a partially ordered set. Show the family of all chains in $P$ is a family of finite character.
\end{Problem}

\begin{Problem}
Suppose $\mathcal{FC}$ is a family of subsets of $X$ that
is of finite character. Order $\mathcal{FC}$ by inclusion. Show every chain in $\mathcal{FC}$ has its union as an upper bound.
\end{Problem}

\begin{Axiom}[Zermelo's Well-ordering Axiom]\label{WOA}
Any set can be well-ordered.
\end{Axiom}

\begin{Axiom}[Kuratowski-Zorn Axiom]\label{KZLemma}
Suppose a partially ordered set $P$ has the property that every chain (i.e. totally ordered subset) has an upper bound in $P$. Then the set $P$ contains at least one maximal element.
\end{Axiom}

\begin{Axiom}[Hausdorff Maximal Principle]\label{HMaxPrinciple}
In any partially ordered set, every totally ordered subset is contained in a maximal totally ordered subset. Here a maximal totally-ordered subset is one that, if enlarged in any way, does not remain totally ordered.
\end{Axiom}

\begin{Axiom}[Teichm\" uller-Tukey Axiom]\label{TukeyAxiom}
Suppose $\mathcal{FC}$ is a family of subsets of a set $X$ and $A\in \mathcal{FC}$. If $\mathcal{FC}$
is of finite character, then there is $A_0\in \mathcal{FC}$ containing $A$ which is maximal ($A_0\subset B\in \mathcal{FC}$ implies $A_0=B$).
\end{Axiom}

\begin{Problem}
Show Axiom of Choice implies Kuratowski-Zorn Axiom.
\end{Problem}

\begin{Problem}
Show Kuratowski-Zorn Axiom  implies Teichm\" uller-Tukey Axiom.
\end{Problem}

\begin{Problem}
Show Teichm\" uller-Tukey Axiom implies Hausdorff Maximal Principle.
\end{Problem}

\begin{Problem}
Show Hausdorff Maximal Principle implies the Axiom of Choice.
\end{Problem}

\section{Ordinal numbers}

Given two well-ordered sets $(X,<_X)$ and $(Y,<_Y)$ we may define an inequality between them using the concept of the pull-back of an ordering. That leads to a partial ordering on any family of ordered sets and it leads to the concept of equivalent well-ordered sets. The interesting aspect of this is that the resulting set of equivalence classes ends up with a well-ordering leading to the concept of \textbf{ordinals}. This section is devoted to developing necessary tools to establish those results.

There is basically one way to define equivalence of well-ordered sets:
\begin{Definition}\label{EquivalenceOfOrdered}
Two well-ordered sets $(X,<_X)$ and $(Y,<_Y)$ are equivalent (notation: $(X,<_X)\sim (Y,<_Y)$) if there is a bijection
$f:X\to Y$ such that the order $<_X$ on $X$ is the pull-back of the order $<_Y$ on $Y$ under $f$.
\end{Definition}

\begin{Problem}
Show the above definition does define an equivalence relation on the family of all ordered subsets of $U$.
\end{Problem}

There are two possible ways of defining inequality between ordered sets. One is analogous to \ref{EquivalenceOfOrdered}:
\begin{Definition}\label{InequalityOfOrdered}
Two well-ordered sets $(X,<_X)$ and $(Y,<_Y)$ are satisfying inequality $(X,<_X)\leq (Y,<_Y)$) if there is an injection
$f:X\to Y$ such that the order $<_X$ on $X$ is the pull-back of the order $<_Y$ on $Y$ under $f$.
\end{Definition}

The other is obtained by strengthening \ref{InequalityOfOrdered}. Namely, we want the image $f(X)$ of $X$ to be an initial subset of $Y$.
\begin{Definition}
A subset $A$ of a well-ordered set $(Y,<_Y)$ is called \textbf{initial} if either $A=Y$ or there is $y\in Y$ such that
$A=(-\infty,y)=\{z\in Y | z < y\}$.
\end{Definition}

\begin{Definition}\label{MinPresFuncDef}
Suppose $(X,<_X)$ and $(Y,<_Y)$
are well-ordered sets.
A function $f\colon X\to Y$ is {\bf minimum-preserving}
if $\min(f(A))=f(\min(A))$ for every non-empty subset $A$ of $X$.
\end{Definition}

\begin{Definition}\label{OrderPresFuncDef}
Suppose $(X,\leq_X)$ and $(Y,\leq_Y)$
are partially ordered sets.
A function $f\colon X\to Y$ is {\bf order-preserving}
if $a\leq_X b$ implies $f(a)\leq_Y f(b)$.
\end{Definition}

\begin{Problem}
Suppose $(X,<_X)$ and $(Y,<_Y)$
are well-ordered sets. Show $f\colon X\to Y$ is minimum-preserving
if and only if it is order-preserving.
\end{Problem}

\begin{Definition}\label{MinOfFunctionsDef}
Suppose $(Y,<_Y)$
is a well-ordered set. Given two functions
$f,g\colon X\to Y$ define their {\bf minimum}
$\min(f,g)$ as $h\colon X\to Y$ so that
$h(x)=\min(f(x),g(x))$.
\end{Definition}

\begin{Problem}
Suppose $(X,<_X)$ and $(Y,<_Y)$
are well-ordered sets. Show the minimum of two
order-preserving functions is order-preserving.
\end{Problem}

\begin{Problem}
Suppose $(Y,<_Y)$
is a well-ordered set. Given a family of functions
$f_s\colon X\to Y$, $s\in S$, define their minimum
$\min\{f_s\}_{s\in S}$.
\end{Problem}

\begin{Problem}
Suppose $(X,<_X)$ and $(Y,<_Y)$
are well-ordered sets. Given a family of order-preserving functions
$f_s\colon X\to Y$, $s\in S$, show their minimum
$\min\{f_s\}_{s\in S}$ is order-preserving.
\end{Problem}

\begin{Problem}
Suppose $(X,<_X)$ and $(Y,<_Y)$
are well-ordered sets.  Given a family of order-preserving injections
$f_s\colon X\to Y$, $s\in S$, show their minimum
$\min\{f_s\}_{s\in S}$ is an injection.
\end{Problem}

Thus one can create the concept of a {\bf minimal injection}
from a well-ordered set $X$ to a well-ordered set $Y$.

\begin{Definition}
Suppose $(X,<_X)$ and $(Y,<_Y)$
are well-ordered sets. An order-preserving function $f:X\to Y$ is a \textbf{minimal injection}
if it is one-to-one and $f\leq g$ for any order-preserving injection $g:X\to Y$.
\end{Definition}

\begin{Problem}
Suppose $X\subset Y$ and the well-ordering of $X$ is inherited from $Y$.
Show the inclusion $i\colon X\to Y$ is minimal if and only if
$X$ is an initial subset of $Y$.
\end{Problem}

\begin{Problem}
Suppose $f\colon X\to Y$ is an order-preserving injection
of well-ordered sets. Show $f$ is minimal if and only if $f(X)$ is an initial
subset of $Y$.
\end{Problem}

\begin{Problem}
Suppose $(X,<_X)$ and $(Y,<_Y)$
are well-ordered sets. Show that if for every $x\in X$ there is an order-preserving
injection from $(-\infty,x]$ to $Y$, then
there is an order-preserving injection
from $X$ to $Y$.
\end{Problem}

\begin{Problem}
Suppose $(X,<_X)$ and $(Y,<_Y)$
are well-ordered sets. Show
there is an order-preserving injection
from one of them to the other whose image is an initial subset.
\end{Problem}

\begin{Definition}
An \textbf{ordinal number} in $U$ is the equivalence class of a well-ordered subset of $U$.
\end{Definition}

\begin{Problem}
Show that the set of ordinals in $U$ is well-ordered.
\end{Problem}

\section{Cardinal numbers}

\begin{Definition}
Two subsets $A$ and $B$ of $U$ are of the same \textbf{cardinality} if there is a bijection between them.
\end{Definition}

\begin{Problem}
Given a natural number $m\in \mathbb N$ find a bijection $f:\mathbb N\to [n,\infty)$.
\end{Problem}

\begin{Problem}
Suppose $X_1\cap Y_1=\emptyset$ and $X_2\cap Y_2=\emptyset$. Show $X_1\cup Y_1$ is of the same cardinality as $X_2\cup Y_2$ if $X_1$ is of the same cardinality as $X_2$ and $Y_1$ is of the same cardinality as $Y_2$.
\end{Problem}

\begin{Problem}
Let $\{X_s\}_{s\in S}$ and $\{Y_s\}_{s\in S}$ be two families consisting of mutually disjoint sets ($X_s\cap X_t=\emptyset$ for all $s\ne t$ and $Y_s\cap Y_t=\emptyset$ for all $s\ne t$).
 Show $\bigcup_{s\in S} X_s$ is of the same cardinality as $\bigcup_{s\in S} Y_s$ if each $X_s$ is of the same cardinality as $Y_s$ for $s\in S$.
\end{Problem}

\begin{Problem}
If $U$ is well-ordered, show that for every two subsets of $U$ there is an injection from one of them to the other.
\end{Problem}

\begin{Theorem}[Schr\" oder-Bernstein]
Suppose $A,B$ are two subsets of $U$. If there are injections from each of them to the other, then they are of the same cardinality.
\end{Theorem}
\begin{proof}
We may assume $A\subset B$ and $f:B\to A$ is an injection.\\
Consider the orbits of $f$. Those are equivalence classes of the following relation on $B$: $b\sim c$
if there are $k,m\ge 1$ such that $f^m(c)=f^k(b)$. Notice that for any $C\subset B$ containing $f(A)$ the orbits of $f|C$ are either equal to orbits of $f$ or they miss a finite number of initial elements of the corresponding orbit of $f$ which is infinite.
\end{proof}

\begin{Definition}
A \textbf{cardinal number} in $U$ is an equivalence class of a subset of $U$ under the relation of having the same cardinality.
\end{Definition}

Notice the set of cardinal numbers in $U$ is partially ordered in view of Schr\" oder-Bernstein's Theorem.
To make the set of cardinal numbers in $U$ totally ordered we need a new axiom.

\begin{Axiom}[Axiom of Comparing Cardinal Numbers]
Given two sets there is an injection of one of them into the other.
\end{Axiom}

Here is another problem in the spirit of Russell's Paradox.
\begin{Problem}
Show that the set of ordinals in $U$ does not admit an injection into $U$ if $U$ is not well-ordered.
\end{Problem}

\begin{Problem}
Show Axiom of Comparing Cardinal Numbers is equivalent to the Well-ordering Axiom.
\end{Problem}

Notice there is a natural function from ordinals in $U$ to cardinal numbers in $U$. If it is a surjection, it has a natural section. 

\begin{Problem}
If $U$ can be well-ordered, then there is order-preserving injection from the set of cardinal numbers in $U$ to the set of ordinals in $U$.
\end{Problem}

\begin{Problem}
If $U$ can be well-ordered, then the set of cardinal numbers in $U$ is well-ordered.
\end{Problem}


\begin{thebibliography}{99}

\bibitem{Dev} Keith Devlin, Introduction to Mathematical Thinking, Amazon 2013.

\bibitem{Hal} Paul Halmos, {\em Naive set theory}, Princeton, NJ: D. Van Nostrand Company, 1960. Reprinted by Springer-Verlag, New York, 1974. ISBN 0-387-90092-6 (Springer-Verlag edition).

\bibitem{Jan} K.J\" anich, \emph{Topology}, Springer-Velag, Undergraduate Texts in Mathematics 1984.

\bibitem{Mun} J.R.Munkres, {\em Topology} (2nd edition), Prentice Hall 2000.

\bibitem{VINK} O.Ya.Viro, O.A.Ivanov, N.Yu.Netsvetaev, V.M.Kharlamov, {\em Elementary Topology Problem Textbook}, AMS 2009.
\end{thebibliography}
\end{document}